\documentclass{amsart}
\usepackage{amssymb,amsmath,amsthm}
\usepackage{hyperref}

\newtheorem{theorem}{Theorem}
\newtheorem{lemma}[theorem]{Lemma}
\newtheorem{corollary}[theorem]{Corollary}
\newtheorem{proposition}[theorem]{Proposition}

\theoremstyle{definition}
\newtheorem{definition}[theorem]{Definition}
\newtheorem{conjecture}[theorem]{Conjecture}
\newtheorem{remark}[theorem]{Remark}

\title[Slope inequalities for the spin geography problem]
{Slope inequalities for the geography problem\\ 
of spin symplectic 4-manifolds}

\author[S. Fushida-Hardy]{Shintaro Fushida-Hardy}
\address{Department of Pure Mathematics, 
University of Waterloo, 
Waterloo, ON, N2L 3G1, Canada} 
\email{sfushida@uwaterloo.ca}

\author[R. Harris]{Robert Harris}
\address{D\'epartement de Math\'ematiques, 
Universit\'e du Qu\'ebec \`a Montr\'eal, 
Montr\'eal, QC, H2X 3Y7, Canada}
\email{harris.robert@uqam.ca}

\author[B. D. Park]{B. Doug Park}
\address{Department of Pure Mathematics, 
University of Waterloo, 
Waterloo, ON, N2L 3G1, Canada} 
\email{bdpark@uwaterloo.ca}

\date{November 7, 2025}

\subjclass[2020]{57K43, 57R55, 11N05}

\keywords{Bogomolov-Miyaoka-Yau inequality, Cram\'er's conjecture, 
geography problem, symplectic $4$-manifolds}

\begin{document}

\begin{abstract}	
We construct infinitely many new pairwise nondiffeomorphic smooth structures on infinitely many closed simply connected spin $4$-manifolds with positive signature. Our construction builds on an infinite family of simply connected complex surfaces of general type due to Roulleau and Urz\'ua that populate points arbitrarily near the Bogomolov-Miyaoka-Yau line in the complex geography plane. 
We also discuss how the conjectural symplectic Bogomolov-Miyaoka-Yau inequality implies that our results are close to being optimal.  
\end{abstract}

\maketitle

\section{Introduction}

By Freedman’s seminal work \cite{freedman}, the homeomorphism type of a closed simply connected oriented smooth $4$-dimensional manifold ($4$-manifold for short) is determined by its intersection form on the second homology group.  
The ``symplectic geography problem'' asks when a symmetric bilinear form (form for short) can be realized as the intersection form of a \emph{symplectic}\/ $4$-manifold.  
The problem has been answered when the form has negative signature (cf.~\cite{ap: odd,ps}).  
In this paper, we will study the open problem of when the signature of the form is nonnegative, with a focus on realizability by \emph{spin}\/ symplectic 4-manifolds.  

Recall from Chapter~1 in \cite{gompf-stipsicz} that a closed simply connected oriented smooth $4$-manifold $M$\/ is spin if and only if its intersection form 
$$I_M:H_2(M;\mathbb{Z})\otimes H_2(M;\mathbb{Z})\to\mathbb{Z}$$ has a block matrix representative $[I_M]=sE_8 \oplus qH$\/ for some integers $s$\/ and $q$, with $s$\/ even and $q\geq 0$, where 
\begin{equation*}
E_8 = \mbox{{\small $\left[
\begin{array}{cccccccc}
2&1&0&0&0&0&0&0\\
1&2&1&0&0&0&0&0\\
0&1&2&1&0&0&0&0\\
0&0&1&2&1&0&0&0\\
0&0&0&1&2&1&0&1\\
0&0&0&0&1&2&1&0\\
0&0&0&0&0&1&2&0\\
0&0&0&0&1&0&0&2
\end{array}
\right]$}} \textnormal{ \ and \ } 
H = \left[\begin{array}{cc}
0&1 \\
1&0
\end{array} 
\right].
\end{equation*}
For such $M$, we have 
\begin{equation}\label{eq: e and sigma}
e(M)=2+8s+2q \quad \mathrm{and \ \ } \sigma(M)=8s,
\end{equation}
where $e(M)$ and $\sigma(M)$ are the Euler characteristic and the signature of $M$, respectively.  
We require $s$\/ to be even because the signature of a closed oriented smooth $4$-manifold is divisible by $16$ according to Rokhlin's theorem (cf.~\cite{rohlin}).  
If the signature of $M$\/ is nonnegative (i.e., if $s\geq 0$), then $b_2^+(M)=8s+q$, and $b_2^-(M)=q$, 
where $b_2^+(M)$ and $b_2^-(M)$ are the dimensions of the maximal positive-definite and maximal negative-definite subspaces of $H_2(M;\mathbb{Z})$ under the intersection form, respectively.  

If a closed simply connected oriented smooth $4$-manifold $M$\/ is symplectic, 
then $M$\/ possesses an almost complex structure.  
By Hirzebruch's signature theorem (cf.~\cite{hirzebruch}), we must have \[e(M)+\sigma(M)=2+2b_2^+(M)\equiv 0\pmod{4}.\]
Thus $b_2^+(M)$ must be odd, and hence $q$ must be odd when $s\geq 0$.  

Now recall that a smooth $4$-manifold $M$\/ is \emph{irreducible}\/ if 
every connected sum decomposition $M= X\#Y$ implies that either $X$\/ or $Y$\/ is homeomorphic to the standard $4$-sphere $S^4$.  
Irreducible $4$-manifolds serve as basic building blocks in the classification of closed simply connected smooth $4$-manifolds up to connected sum operations.  
To state our results more efficiently, we recall the following definition from \cite{ap: spin}.  

\begin{definition}\label{definition: infinity squared}
We say that a smooth oriented $4$-manifold $M$\/ has $\infty$-\emph{property}\/ if there exist infinitely many pairwise nondiffeomorphic irreducible smooth $4$-manifolds
that are all homeomorphic to $M$.  
We say that $M$\/ has $\infty^2$-\emph{property}\/ if there exist infinitely many pairwise nondiffeomorphic irreducible symplectic $4$-manifolds and infinitely many pairwise nondiffeomorphic irreducible nonsymplectic $4$-manifolds, all of which are homeomorphic to $M$.  
We also say that a symmetric bilinear form has $\infty$-\emph{property}\/ or $\infty^2$-\emph{property}\/ if it is the intersection form of a 
simply connected smooth $4$-manifold with $\infty$-\emph{property}\/ or $\infty^2$-property, respectively.
\end{definition}

It is immediate that $\infty^2$-property implies $\infty$-property.  
The converse is false in general.  
For example, if any simply connected $4$-manifold $M$\/ has $\infty^2$-property and $\sigma(M)$ is odd, then the (non-spin) manifold $\overline{M}$ with the reverse orientation will inherit $\infty$-property from the $\infty$-property on $M$, but $\overline{M}$\/ will not admit any symplectic structure.  
In fact, $\overline{M}$\/ will not admit any almost complex structure since
$b_2^+(\overline{M})=b_2^-(M)=b_2^+(M)-\sigma(M)$ will be even. 

There is much less known regarding the converse in the spin case (and more generally, the even signature case). For example, if $E(2)$ denotes a complex K$3$ surface, then $E(2)$ has $\infty^2$-property (see e.g.~\cite{fs}). Then similar to above, $\overline{E(2)}$ (with $\sigma(\overline{E(2)})=-\sigma(E(2))=16$) will also have $\infty$-property, but it is unknown at the moment whether $\overline{E(2)}$ supports any symplectic structure.  
Hence we do not know at the moment whether the positive-signature spin $4$-manifold $\overline{E(2)}$ has $\infty^2$-property.  

Given a nonnegative even integer $s$, 
it was shown in \cite{jpark: spin} that $sE_8 \oplus qH$\/ has $\infty^2$-property when the odd integer $q$\/ is larger than some constant that depends only on $s$.  
In other words, for a fixed nonnegative value of the signature, our form will have $\infty^2$-property if its rank is high enough.  
Accordingly, the following notation was introduced in \cite{ap: spin}.  

\begin{definition}\label{definition: Lambda}
(cf.~Definition~3 in \cite{ap: spin})
For an even integer $s\geq 0$, 
let $\Lambda_s$ denote the smallest positive odd integer such that the
symmetric bilinear form
$sE_8 \oplus qH$\/ has $\infty^2$-property for every odd integer $q\geq \Lambda_s$.  
\end{definition}

Roughly speaking, $\Lambda_s$ is the ``smallest'' number of $H$\/ summands that 
a simply connected spin $4$-manifold $M$\/ with signature $\sigma(M)=8s\geq 0$ must have in its intersection form $I_M$ in order for $M$\/ to have $\infty^2$-property.
As $s\to\infty$, an asymptotic upper bound 
\begin{equation}\label{eq: apu bound}
\Lambda_s \leq 8s + O(s^{6/7}) 
\end{equation}
was proved in \cite{apu}.  
The main goal of this paper is to provide a smaller (and hence better) asymptotic upper bound on $\Lambda_s$.  

More precisely, we will prove (Section~\ref{section: asymptotic inequalities}, Corollary~\ref{corollary: Lambda}) that as $s\to\infty$, 
\begin{equation}
\Lambda_s \leq 8s
+O\big(s^{4/5}\,d(s)\big),
\end{equation}
where $d(s)$ (defined in Section~\ref{section: asymptotic inequalities}, Equation~(\ref{eq: d(s)})) measures prime gaps. For example, with the assumption of the famous Cram\'er's conjecture on the size of prime gaps, we have (Section~\ref{section: asymptotic inequalities}, Proposition~\ref{proposition: new upper bound})
\begin{equation}\label{eq: new upper bound}
 \Lambda_s \leq 8s + O\big(s^{4/5}(\log s)^2\big),  
\end{equation}
as $s\to\infty$.

We note that $6/7= 0.\overline{857142}$, so (\ref{eq: new upper bound}) is strictly stronger than (\ref{eq: apu bound}), provided that Cram\'er's conjecture holds.  We refer to Remark~\ref{remark: alternative bound} for a way to obtain an alternative upper bound without resorting to Cram\'er's conjecture.

Along the way, we will also construct infinitely many new simply connected spin $4$-manifolds having $\infty^2$-property (Section~\ref{section: symplectic 4-manifolds}, Lemma~\ref{lemma: W_n}).  
The new basic building blocks that we use are the complex surfaces recently constructed by Roulleau and Urz\'ua in \cite{ru}.  
These surfaces and certain curves in them are reviewed in Sections~\ref{section: complex surfaces} and \ref{section: Sigma_n in X_n}.  
Finally, we will derive a conditional lower bound on $\Lambda_s$:
\[
\Lambda_s\geq 8s -1
\]
in the last section (Proposition~\ref{proposition: lower bound}), assuming that a ``symplectic'' version of the famous Bogomolov-Miyaoka-Yau inequality holds.  
Consequentially, the dominant term $8s$\/ in the upper bound (\ref{eq: new upper bound}) cannot be improved any further when the symplectic BMY inequality holds.

\subsection*{Acknowledgments}
The third author was partially supported by an NSERC Discovery grant 
RGPIN--2024--04401.
The authors thank the organizers of the Topology Session of 
the 2025 Canadian Mathematical Society Winter Meeting for the opportunity to present this work.

\section{Complex surfaces of Roulleau and Urz\'ua}
\label{section: complex surfaces}

In Section~5 of \cite{ru}, Roulleau and Urz\'ua had constructed an interesting infinite family of complex surfaces $\{X_n\}$.  
The topological invariants of $X_n$ were not computed explicitly in \cite{ru}, and so we will actually compute some of them in this section.  
We will use the same set of notations as in \cite{ru} to very succinctly describe these surfaces.  
As such, all page numbers in this section will refer to \cite{ru}.  
Let $p\geq 5$ be a prime number.  
We choose the parameters $\alpha=1$, $\beta=0$, and $d=1$ as in Example~5.7 of \cite{ru}.  
From these choices, we get $n=12\alpha p=12p$ according to a formula on p.~294.  
By Propositions~5.4 and 5.5 of \cite{ru}, $X_n=X_{12p}$ are simply connected and spin for all prime numbers $p\geq 5$.  

From formulas on p.~298, we get 
\begin{align*}
t_2 &= 13824 \beta^2\alpha^2 p^4+1152 \beta^4 p^4 +4608 d\alpha^2 p^2 
+768 d\beta^2 p^2 +32 d^2 -100d \\\
&= 4608 p^2-68,\\ 
\bar{c}_1^2 &= n^4+2t_2-40d-48\\
&= 20736 p^4 + 9216 p^2 -224, \\ 
\bar{c}_2 &= \frac{n^4}{3}+t_2-16d-12\\
&=6912 p^4 + 4608 p^2-96.  
\end{align*}
From formulas on p.~299, we also get 
\begin{align*}
t_{2,1} &= 384\beta^4 p^4 +4608\alpha^2\beta^2 p^4 +2304d\alpha^2p^2 
-52d+384d\beta^2p^2+16d^2\\
&= 2304 p^2-36, \\ 
t_{2,2} &= 768\beta^4 p^4 + 9216\alpha^2\beta^2p^4 +2304d\alpha^2p^2
-48d+384d\beta^2p^2+16d^2\\
&= 2304 p^2 -32,
\end{align*}
and 
\begin{align*}
\sum_{i<j}c(q_{i,j},4p)A_i\cdot A_j &= 
\frac{4p-1}{2p}\,t_{2,1}+\frac{2p^2+1}{2p}\,t_{2,2}\\[-7pt]
&=2304 p^3 +4608 p^2 -32 p -72 + \frac{2}{p},\\
\sum_{i<j}l(q_{i,j},4p)A_i\cdot A_j &= 
(4p-1)t_{2,1}+3t_{2,2}\\[-7pt]
&=9216 p^3 + 4608 p^2 -144 p -60.
\end{align*}
From formulas on p.~298, the Chern numbers of our surfaces $X_n=X_{12p}$ are:  
\begin{align*}
c_1^2(X_n) &= 4p \bar{c}_1^2-2\bigg(t_2+2\sum_j(g(A_j)-1)\bigg)
+\frac{1}{4p}\sum_j A_j^2
-\sum_{i<j}c(q_{i,j},4p)A_i\cdot A_j\\
&=4p(20736 p^4 + 9216 p^2 -224)-2(4608 p^2-68) \\[3pt]
&\hspace{4mm} -\bigg(2304 p^3 +4608 p^2 -32 p -72 + \frac{2}{p}\bigg)
-4\sum_{j}(g(A_j)-1)+\frac{1}{4p}\sum_j A_j^2\\
&= 82944p^5+34560p^3-13824p^2-864p+208-\frac{2}{p}\\
&\hspace{4mm} -4\sum_{j}(g(A_j)-1)+\frac{1}{4p}\sum_j A_j^2,
\end{align*}
and 
\begin{align*}
e(X_n) &= c_2(X_n)= 
4p\,\bar{c}_2-\bigg(t_2+2\sum_j(g(A_j)-1)\bigg)
+\sum_{i<j}l(q_{i,j},4p)A_i\cdot A_j 
\\
&=4p(6912 p^4 + 4608 p^2-96)
-(4608 p^2-68)-2\sum_{j}(g(A_j)-1)\\
&\hspace{4mm} +9216 p^3 + 4608 p^2 -144 p -60\\[2pt]
&= 27648 p^5 + 27648 p^3 -528 p + 8-2\sum_{j}(g(A_j)-1).
\end{align*}

By a formula on p.~295, the branch divisor 
$\sum_j A_j$ is equal to 
\begin{equation*}
\sum_{i=0,1,\zeta,\infty} \hspace{-2mm}\mathcal{E}_i 
\hspace{2mm}+\sum_{i=0,1,\zeta,\infty} \hspace{-2mm}\mathcal{E}'_i
\hspace{2mm}+\sum_{i=0,1,\zeta,\infty} \hspace{-2mm}N_i
\hspace{2mm}+\hspace{2mm}\sum_{i=1}^{8d} L_i.
\end{equation*}
According to a formula on p.~294, the number of general fibers of $\pi'_i$ in the $\mathcal{E}'_i$ divisor is now $8\beta^2 p^2=0$, and thus we can conclude that each $\mathcal{E}'_i$ divisor is equal to zero.  
It follows that 
\begin{equation}\label{eq: branch divisor}
\sum_j A_j \hspace{2mm} =
\sum_{i=0,1,\zeta,\infty} \hspace{-2mm}\mathcal{E}_i 
\hspace{2mm}+\sum_{i=0,1,\zeta,\infty} \hspace{-2mm}N_i
\hspace{2mm}+\hspace{2mm}\sum_{i=1}^{8} L_i.
\end{equation}
Recall from p.~294 that each component of the divisor $\mathcal{E}_i$ has genus $1$, each component of the divisor $N_i$ has genus $0$, and each divisor $L_i$ has genus $0$.  
Since $N=\sum_{i=0,1,\zeta,\infty} N_i=\sum_{i=0,1,\zeta,\infty}
(N_{i,1}+N_{i,2}+N_{i,3})$ has $12$ components, we have 
\begin{equation*}
\sum_{j}(g(A_j)-1) = 12(-1)+8(-1)=-20.  
\end{equation*}
Therefore, we conclude that 
\begin{equation}\label{eq: e(X_n)}
e(X_n) = 27648 p^5 + 27648 p^3 -528 p + 48.
\end{equation}

Next recall from p.~292 that the divisor $\mathcal{H}'_n=\sum_{i=0,1,\zeta,\infty} \mathcal{E}_i$ consists of strict transforms of $4(n\bar{n}-3)/3=192p^2-4$ tori.  
We note that $\mathcal{H}'_n$ had $$4(n\bar{n}-3)=576 p^2-12$$ triple points (i.e., points where three tori met) and 
$$(n\bar{n}-3)(n\bar{n}-9)/3=6912p^4 - 576p^2 +9$$ quadruple points (i.e., points where four tori met) which were blown up.  
When we blow up a triple point, the resulting exceptional divisor will add $-1$ to exactly three $A_j^2$ terms corresponding to the three tori that meet at the triple point.  
Similarly, we will add four $-1$'s to the sum $\sum_{j}A_j^2$ when we blow up a quadruple point.
As observed on p.~291, each of these tori had self-intersection $-3$ before these blow-ups.  
Now each $N_{i,j}$ has self-intersection $-1$ and each $L_i$ has self-intersection $+1$.  
Combining all these facts, we compute that 
\begin{align*}
\sum_j A_j^2 &= -3(192p^2-4)-3(576 p^2-12)-4(6912p^4 -576p^2 +9)
+12(-1)+8\cdot 1\\[-8pt]
&= -27648p^4 +8.
\end{align*}
Therefore, we conclude that 
\begin{align}
c_1^2(X_n) &= 82944p^5+34560p^3-13824p^2-864p+208-\frac{2}{p}\nonumber\\[-3pt]
&\hspace{4mm} -4(-20)+\frac{1}{4p}(-27648p^4 +8)
\label{eq: c_1^2(X_n)}\\[2pt]
&=82944p^5+27648p^3-13824p^2-864p+288.\nonumber
\end{align}
It follows from (\ref{eq: e(X_n)}) and (\ref{eq: c_1^2(X_n)}) that 
\begin{equation}\label{eq: sigma(X_n)}
\sigma(X_n)=\frac{1}{3}(c_1^2(X_n)-2e(X_n))
=9216p^5 -9216p^3 -4608p^2 +64p +64.
\end{equation}

\begin{remark}\label{remark: limit of c_1^2/c_2}
Note that $\sigma(X_n)$ is divisible by $16$ as $X_n$ is spin.  
When $p=5$, formulas (\ref{eq: e(X_n)}) and (\ref{eq: c_1^2(X_n)}) 
give $89853408$ and $262306368$, 
respectively, which agree with the computations provided in Example 5.7 of \cite{ru}.  
Finally, we note that 
\[
\lim_{p\to \infty} \frac{c_1^2(X_n)}{e(X_n)}
=\frac{82944}{27648}=3.
\]
\end{remark}

\section{Complex curve in $X_n$}
\label{section: Sigma_n in X_n}

In this section, we exhibit a smooth complex curve $\Sigma_n$ lying in $X_n$ such that its self-intersection is an even positive integer.
As in the previous section, we shall use the same notations as in \cite{ru}, and all page numbers refer to \cite{ru} as well.
We recall that there is a sequence of holomorphic maps:  
\[
X_n \stackrel{f}{\longrightarrow} Y_n 
\stackrel{\sigma_n}{\longrightarrow} Z_n 
\stackrel{\varphi_n}{\longrightarrow} H
\stackrel{\tau}{\longrightarrow}\mathbb{CP}^2.
\]
Here, $f=f_1\circ f_2\circ f_3$ is the composition of a $4p$-fold cyclic branched covering map $f_1$, 
its normalization $f_2$, and finally the corresponding minimal resolution $f_3$ (see p.~295).  
The other maps are all sequences of blow-ups.  More specifically, 
$\tau$ consists of blow-ups at $12$ points (see p.~291), 
$\varphi_n$ consists of blow-ups at $(n\bar{n}-3)(n\bar{n}-9)/3$ points (see p.~293), and 
$\sigma_n$ consists of blow-ups at $4(n^2-3)$ points (see p.~295).

Now let us choose a generic line $L\subset \mathbb{CP}^2$ so that it avoids all the blow-up points 
of $\tau$, $\varphi_n$ and $\sigma_n$.  
Let us continue to denote its proper transform in $Y_n$ by $L$.  
The self-intersection of $L$\/ in $Y_n$ is still $[L]^2=1$ 
since $L$\/ is disjoint from all the exceptional divisors of the blow-ups.  

Next let $\widetilde{L}=f_1^{-1}(L)$ denote the preimage of $L$\/ under the branched covering map.  
Since the degree of the covering map is $4p$, its self-intersection is $[\widetilde{L}]^2=4p[L]^2=4p$.  
(After a sequence of $4p$\/ blow-downs, a generic fiber of the fibration map $h$\/ in the proof of Proposition~5.4 in \cite{ru} 
will turn into $\widetilde{L}$.)  
If $\cup_j A_j$ denotes the branch locus of the branched covering map $f_1$
as in (\ref{eq: branch divisor}), 
we will write $\widetilde{L}_0=\widetilde{L}\setminus(\widetilde{L}\cap(\cup_j A_j))$.  
The restriction of $f_1$ to $\widetilde{L}_0$ is an unbranched cyclic covering map, and hence $\widetilde{L}$ is connected if the induced homomorphism into the deck transformation group 
\[
\rho:\pi_1(\widetilde{L}_0)
\longrightarrow \mathbb{Z}/(4p\mathbb{Z})
\]
is surjective.  

Now recall that the branch multiplicity $\nu_j$ of $A_j$ is either $3$ or $3(2p-1)$ (see p.~295), which are all coprime to $4p$.
By the local description of cyclic branched covers 
(cf.~the proof of Claim 3.13 in \cite{ev}), 
 a local model at a generic point of $A_j$ is given by 
\[ 
\{ (x,t)\in\mathbb{C}^2 \mid
t^{4p}=ux^{\nu_j}\},
\]
where $u$\/ is some unit constant.  
Since $\gcd(4p,\nu_j)=1$, the above polynomial is irreducible by the Newton-polygon criterion, and hence 
(the preimage of) each $A_j$ is irreducible in the cyclic branched cover,
and each $A_j$ is totally ramified of index $4p$.

Consider the branch component $A_j=L_1$, which is also a preimage of a line in $\mathbb{CP}^2$ (see p.~294).  
Since $\widetilde{L}\cdot L_1 = L\cdot L_1=1$, a meridian circle $\mu_1$ for $L_1$ will represent an element of  $\pi_1(\widetilde{L}_0)$.  
Since $L_1$ is totally ramified, the image $\rho(\mu_1)$ will generate all of 
$\mathbb{Z}/(4p\mathbb{Z})$.  
Thus we can conclude that $\widetilde{L}$ is connected.  

To compute the genus $g(\widetilde{L})$, we will compute the Euler characteristic $e(\widetilde{L})$ using the inclusion–exclusion principle or the Riemann-Hurwitz formula.  
(A similar computation is carried out in the proof of Corollary~6.4 on p.~303 of \cite{ru}.)
It follows from (\ref{eq: branch divisor}) that 
\begin{equation*}
\widetilde{L}\cdot \Big(\sum_j A_j\Big) \hspace{2mm}  = 
 \sum_{i=0,1,\zeta,\infty} \hspace{-2mm} \widetilde{L}\cdot\mathcal{E}_i 
\hspace{2mm}+\sum_{i=0,1,\zeta,\infty} \hspace{-2mm} \widetilde{L}\cdot N_i
\hspace{2mm}+\hspace{2mm}\sum_{i=1}^{8} \widetilde{L}\cdot L_i .
\end{equation*}
As before, we have $\widetilde{L}\cdot L_i=1$ for each $i$.  
Since a generic line from $\mathbb{CP}^2$ will miss all the exceptional spheres of the blow-ups in $\tau$, $\widetilde{L}\cdot N_i=0$ for each $i$.  

For each generic torus fiber $F_i$ in the divisor $\mathcal{E}_i$ of the branch locus,  
its image $(\tau\circ\varphi_n\circ \sigma_n\circ f_1)(F_i)$ in $\mathbb{CP}^2$ is a smooth torus as $\tau\circ\varphi_n\circ \sigma_n$ is just a sequence of blow-down maps.  
Since a smooth torus in $\mathbb{CP}^2$ is homologous to $3L$, we conclude that  $\widetilde{L}\cdot F_i=L\cdot 3L = 3$.  
It follows that 
\[
\widetilde{L}\cdot \Big(\sum_j A_j\Big) = 3(192p^2-4) + 12\cdot 0 + 8\cdot 1
= 576p^2-4.
\]
Therefore, we have 
\begin{align*}
e(\widetilde{L})&=4p(2-(576p^2-4))+576p^2-4
=4p\cdot 2- (4p-1)(576p^2-4)\\
&=-2304p^3+576p^2+24p-4
\end{align*}
so that 
$$g(\widetilde{L})=\frac{1}{2}(2-e(\widetilde{L}))=
1152p^3-288p^2-12p+3.
$$

Finally, we define $\Sigma_n$ to be the preimage $(f_2 \circ f_3)^{-1}(\widetilde{L})$ in $X_n$.  
Recall that the composition of the normalization map $f_2$ and the minimal resolution map $f_3$ amounts to contracting Hirzebruch–Jung configurations of spheres (see p.~296).  
Thus the composition $f_2 \circ f_3$ restricts to the identity map on a Zariski open subset of $X_n$ and does not 
change the genus nor the self-intersection number of a generic curve.  
Therefore, we can conclude that: 
\begin{align}
[\Sigma_n]^2&=[\widetilde{L}]^2=4p, \label{eq: [Sigma_n]^2}\\ g(\Sigma_n)&=g(\widetilde{L})=1152p^3-288p^2-12p+3. \label{eq: g(Sigma_n)} 
\end{align}

\section{Infinite family of symplectic $4$-manifolds}
\label{section: symplectic 4-manifolds}

The complex surfaces $X_n$ in Section~\ref{section: complex surfaces} may not contain any homologically nontrivial smooth tori with self-intersection $0$, which are crucial ingredients in applying existing techniques from \cite{ap: spin, apu}.  
We bypass this difficulty by performing symplectic normal sum (cf.~\cite{gompf, mccarthy-wolfson}) of $X_n$ with homotopy elliptic surfaces constructed by Fintushel and Stern in \cite{fs}.  

Given a positive integer $r$, let $E(r)$ be a simply connected complex elliptic surface with $e(E(r))=12r$ and $\sigma(E(r))=-8r$. 
We can assume that the elliptic fibration structure on $E(r)$ has a singular cusp fiber.  
Let $K$\/ be a fibered knot in the $3$-sphere $S^3$ with genus $g(K)$.  
Let $E(r)_K$ denote the result of performing a knot surgery along a smooth torus fiber of $E(r)$ in the sense of \cite{fs}.  
It was shown in \cite{fs} that $E(r)_K$ is a simply connected symplectic $4$-manifold that is homeomorphic to $E(r)$.  
If $r$\/ is even, then $E(r)$ is spin, and thus $E(r)_K$ is also spin since they all have isomorphic intersection forms by Freedman's work \cite{freedman}.   
By gluing together a punctured sphere section of the elliptic fibration and a Seifert surface of the knot $K$, we obtain a genus $g(K)$ symplectic submanifold $\Sigma_K$ inside $E(r)_K$ with self-intersection $[\Sigma_K]^2=-r$.  

Now we perform symplectic normal sum (cf.~\cite{gompf, mccarthy-wolfson}) of the symplectic pairs $(X_n,\Sigma_n)$ and $(E(r)_K,\Sigma_K)$, 
where $g(\Sigma_n)=g(K)$ and $[\Sigma_n]^2=r\equiv 0 \pmod{2}$. 
We define 
\[
W_n:=\big(X_n\setminus\nu(\Sigma_n)\big)
\cup_{\varphi}\big(E(r)_K\setminus\nu(\Sigma_K)\big),
\]
where $\nu(\Sigma_n)$ and $\nu(\Sigma_K)$ are tubular neighborhoods of $\Sigma_n$ and $\Sigma_K$, respectively, and 
$\varphi:\partial (\nu(\Sigma_n))\to \partial (\nu(\Sigma_K))$ is an orientation-reversing gluing map on the boundaries.  
Using Equations (\ref{eq: e(X_n)}) and (\ref{eq: sigma(X_n)}) from 
Section~\ref{section: complex surfaces} with Equations (\ref{eq: [Sigma_n]^2}) and (\ref{eq: g(Sigma_n)}) from 
Section~\ref{section: Sigma_n in X_n}, we compute that:   
\begin{align}
e(W_n)&=e(X_n)+e(E(r)_K) - 2e(\Sigma_n) = e(X_n)+12r+4g(\Sigma_n)-4 \label{eq: e(W_n)}\\
&=27648 p^5 + 32256 p^3 -1152p^2 -528 p + 56 ,\nonumber\\
\sigma(W_n)&=\sigma(X_n)+\sigma(E(r)_K) =\sigma(X_n)-8r \label{eq: sigma(W_n)}\\
&=9216p^5 -9216p^3 -4608p^2 +32p +64,\nonumber\\
c_1^2(W_n) &= 2e(W_n) + 3\sigma(W_n) \label{eq: c_1^2(W_n)}\\
&= 82944p^5 + 36864p^3 - 16128p^2 - 960p + 304,\nonumber
\end{align}
where $n=12p$ and $p\geq 5$ is a prime number.  
We observe from (\ref{eq: e(W_n)}) and (\ref{eq: c_1^2(W_n)}) that
\begin{equation}\label{eq: limit of c_1^2/c_2}
\lim_{p\to\infty}\frac{c_1^2(W_n)}{e(W_n)}=\frac{82944}{27648}=3, 
\end{equation}
and for all prime numbers $p\geq 5$, 
\begin{equation}\label{eq: below BMY line}
c_1^2(W_n)-3e(W_n)= -59904p^3 - 12672p^2 + 624p +136<0.
\end{equation}

\begin{lemma}\label{lemma: W_n}
The symplectic\/ $4$-manifold $W_n$ is simply connected, spin, irreducible, and has\/ $\infty^{2}$-property\/ 
$($cf.\ Definition\/ $\ref{definition: infinity squared})$.  
Moreover, $W_n$ contains\/ $r$ disjoint symplectic torus submanifolds\/ $T_j$\/ $(j=1,\dots,r)$ of self-intersection\/ $0$ 
that represent\/ $r$ linearly independent homology classes 
and satisfy\/ $\pi_1(W_n\setminus \nu (\bigcup_{j=1}^{r}T_j))=1$.  
\end{lemma}

\begin{proof}
Our proof is similar to the proof of Theorem~3.2 in \cite{apu}.  
First we will prove that $W_n$ is simply connected.  
By Lemma~3.1 in \cite{apu}, we have $\pi_1(E(r)_K\setminus\nu(\Sigma_K))=1$.
Seifert-Van Kampen theorem implies that 
\begin{equation*}
\pi_{1}(W_n) \cong \frac{\pi_1(X_n\setminus \nu (\Sigma_n))}{\big\langle\imath_{\ast}\big(\pi_1(\partial(\nu(\Sigma_n)))\big)\big\rangle},
\end{equation*}
where $\langle\imath_{\ast}(\pi_1(\partial(\nu(F_n))))\rangle$ is the normal subgroup of $\pi_1(X_n\setminus \nu (F_n))$ generated by the image of $\pi_1(\partial(\nu(F_n)))$ under the homomorphism induced by the inclusion map $\imath: \partial(\nu(\Sigma_n))\hookrightarrow X_n$.

Note that $\partial(\nu(\Sigma_n))$ is a circle bundle over $\Sigma_n$ with Euler number $r$.  
It is well known (cf.~Proposition~10.4 in \cite{fm}) that we have a presentation
\begin{equation*}
\pi_1(\partial(\nu(\Sigma_n))) = \big\langle \alpha_i, \beta_i, \mu \;\big|\; 
{\textstyle \prod_{i=1}^{g(\Sigma_n)}}[\alpha_i,\beta_i]=\mu^{r},\, 
\alpha_i \mu \alpha_i^{-1} = \mu,\, 
\beta_i\mu\beta_i^{-1} = \mu \big\rangle,
\end{equation*}
where the index $i$\/ ranges over $1,\dots,g(\Sigma_n)$.  
Here, $\mu$\/ is represented by a fiber circle which is a meridian of $\Sigma_n$, 
and $\alpha_i,\beta_i$ are the parallel push-offs of the standard generators of $\pi_1(\Sigma_n)$.  
It follows that there is a surjective homomorphism
\begin{equation}\label{eq: surjection onto pi_1(W_n)}
\frac{\pi_1(X_n\setminus \nu (\Sigma_n))}{\langle\imath_{\ast}(\mu)\rangle} \longrightarrow \pi_1(W_n).
\end{equation} 
Since $\pi_{1}(X_{n})=1$, we know that $\pi_1(X_n\setminus \nu (\Sigma_n))$ is normally generated by the meridians of $\Sigma_n$.  
But every meridian of $\Sigma_n$ is conjugate to $\imath_{\ast}(\mu)$ in $\pi_1(X_n\setminus \nu (\Sigma_n))$.  
Hence we must have $\pi_{1}(X_{n}\setminus\nu(\Sigma_n))/\langle\imath_{\ast}(\mu)\rangle=1$, and (\ref{eq: surjection onto pi_1(W_n)}) implies that $\pi_{1}(W_n)=1$ as well.  

Since $X_{n}$ and $E(r)_{K}$ are both spin, $W_n$ can also be given a spin structure 
according to Proposition~1.2 in \cite{gompf}.  
The irreducibility of $W_n$ follows immediately from Lemma~2 in \cite{ap: spin}.  

Next we recall from \cite{gompf-mrowka} that the K$3$ surface $E(2)$ contains three disjoint copies of Gompf nuclei.  
Since $E(r)$ can be viewed as the symplectic normal sum of $r/2$ copies of $E(2)$ glued along smooth torus fibers in Gompf nuclei, the homotopy elliptic surface $E(r)_{K}$ contains $2(r/2)=r$\/ Gompf nuclei of $E(2)$ that are all disjoint from $\nu(\Sigma_{K})$.  
Let $N_j$ ($j=1,\dots,r$) denote these Gompf nuclei that are contained in $(E(r)_{K} \setminus \nu(\Sigma_{K}))\subset W_n$, 
and let $T_j$ be a smooth torus fiber in $N_j$.  
By perturbing the symplectic form on the corresponding $E(2)$ part if necessary, 
we can always arrange every $T_j$ to be a symplectic submanifold of $W_n$.  
Since $T_j$ transversely intersects once a sphere section of $N_j$ with self-intersection $-2$, 
every meridian of $T_j$ is nullhomotopic.  
It follows that $\pi_1(W_n\setminus \nu (\bigcup_{j=1}^{r}T_j))=\pi_1(W_n)=1$.

Finally, to prove that $W_n$ has $\infty^2$-property, we apply the Fintushel-Stern knot surgery in \cite{fs} to $W_n$ one more time.  
By performing a second knot surgery in one of the above nuclei, say $N_1$, with another knot $K'\subset S^3$, we obtain an irreducible $4$-manifold $W_{n,K'}$ that is homeomorphic to $W_n$.  
By varying our choice of the second knot $K'$, we can realize infinitely many pairwise nondiffeomorphic $4$-manifolds, either symplectic or nonsymplectic. 
\end{proof}

\section{Asymptotic inequalities}
\label{section: asymptotic inequalities}

Let us define $\chi_h(M) = \frac{1}{4}(e(M)+\sigma(M))$, 
which equals the holomorphic Euler characteristic when $M$\/ is a complex surface.
We start by stating the following useful theorem from \cite{ap: spin}.  

\begin{theorem}\label{theorem: spin wedge}
 {\rm (cf.~Theorem~9 in \cite{ap: spin} and Remark~4.2 in \cite{apu})} 
Let\/ $Z$ be a closed spin symplectic\/ $4$-manifold that contains a symplectic torus\/ $T$ of self-intersection\/ $0$.  
Let\/ $\nu (T)$ be a tubular neighborhood of\/ $T$ and\/ $\partial(\nu (T))$ its boundary.  
Suppose that the homomorphism\/ $\pi_1(\partial(\nu (T))) \rightarrow \pi_1(Z\setminus\nu (T))$ induced by the inclusion is trivial.  
Then for any pair of integers\/ $(\chi, c)$ satisfying
\begin{equation}\label{eq: chi c inequalities}
0\leq c \leq 8\chi  \quad 
\text{and\/} \quad  
c-8\chi \equiv 0 \pmod{16}, 
\end{equation}
there exists a closed spin symplectic\/ $4$-manifold\/ $Y$ with\/ $\pi_1(Y)=\pi_1(Z)$, 
\begin{equation*}
\chi_h(Y)=\chi_h(Z)+\chi  \quad \text{and\/} \quad 
c_1^2(Y)=c_1^2(Z)+c .   
\end{equation*}
\end{theorem}

We apply Theorem~\ref{theorem: spin wedge} to the $4$-manifolds $W_n$ in Section~\ref{section: symplectic 4-manifolds} and  
prove our main technical theorem below.  

\begin{theorem}\label{theorem: bound}
Let $s\geq 0$ be an even integer.  If\/ $n$ is a positive integer satisfying 
\begin{equation}\label{eq: s inequality}
s\leq \frac{1}{8} \sigma(W_n),
\end{equation} 
then we have  
\begin{equation}\label{eq: Lambda inequality}
\Lambda_s \leq - 10s - 1 
+ \frac{1}{4}c_1^2(W_n).
\end{equation} 
\end{theorem}

\begin{proof}
We apply Theorem~\ref{theorem: spin wedge} to the symplectic pair $(Z,T)=(W_n,T_2)$, 
where $T_2$ is a torus fiber of Gompf nucleus $N_2$ in the 
$(E(r)_K\setminus\nu(\Sigma_K))$ half of $W_n$ 
as in the proof of Lemma~\ref{lemma: W_n}.  
For each $(\chi, c)$ satisfying (\ref{eq: chi c inequalities}), we obtain a closed simply connected spin symplectic 4-manifold $Y$\/ with
\begin{equation*}
e(Y) = e(W_n) + 12\chi - c 
\quad \text{and} \quad
\sigma(Y) = \sigma(W_n) + c - 8\chi .
\end{equation*}
The intersection form of $Y$\/ is represented by $s E_8 \oplus q H$, where 
\begin{align*}
s &= \frac{1}{8}\sigma(Y) = \frac{1}{8}(\sigma(W_n)+c-8\chi) 
\equiv 0 \pmod{2} ,\\[2pt] 
q &= b_2^-(Y) = \frac{1}{2}(e(W_n)-\sigma(W_n)) + 10\chi -c -1
\equiv 1 \pmod{2} .
\end{align*}
Since $c-8\chi\leq 0$ in (\ref{eq: chi c inequalities}), 
we must have $s\leq \frac{1}{8}\sigma(W_n)$, 
which is exactly (\ref{eq: s inequality}).  
Solving for $(\chi, c)$ in terms of $s$, $q$\/ and $n$, we get
\begin{align*}
\chi & = \frac{1}{2}(8s+q+1) -\chi_h(W_n) ,\\[2pt]
c & = 40s + 4q + 4 - c_1^2(W_n).
\end{align*}
Since $c\geq 0$ in (\ref{eq: chi c inequalities}), we must have 
$q\geq -10s -1 + \frac{1}{4}c_1^2(W_n)$.   

Now recall from \cite{ap: spin} that $Y$\/ is the symplectic normal sum of $(W_n,T_2)$ and 
another suitable symplectic pair.  
Since we can perform a knot surgery in the Gompf nucleus
\begin{equation*}
N_1 \,\subset\, \big(E(r)_{K} \setminus (\nu(\Sigma_{K})\cup N_2) \big) 
\,\subset\, (W_n\setminus \nu(T_2)) \,\subset\, Y, 
\end{equation*}
$Y$\/ has $\infty^2$-property just like $W_n$. 
Therefore, for every even integer $s$\/ satisfying 
$0\leq s\leq \frac{1}{8}\sigma(W_n)$, we have 
$\Lambda_s \leq -10s -1 + \frac{1}{4}c_1^2(W_n)$, 
which is exactly (\ref{eq: Lambda inequality}).
\end{proof}

Next recall from Equations (\ref{eq: sigma(W_n)}) and (\ref{eq: c_1^2(W_n)}) in Section~\ref{section: symplectic 4-manifolds} that 
\begin{align*}
\sigma(W_{12p})&=9216p^5 -9216p^3 -4608p^2 +32p +64,\\
c_1^2(W_{12p})&=82944p^5 + 36864p^3 - 16128p^2 - 960p + 304,    
\end{align*}
where $n=12p$\/ and $p\geq 5$ is a prime number.  
Note that both invariants are degree~5 polynomials in $p$.  
We recall that the $4$-manifold building blocks in \cite{apu} all have topological invariants that are degree~7 polynomials (see p.~61 of \cite{apu}), and thus their invariants are comparatively larger than the corresponding invariants of $W_{12p}$.  
Hence applying Theorem~\ref{theorem: bound} to $\{W_{12p}\mid \mathrm{prime\ } p\geq 5\}$ will further yield \emph{infinitely many}\/ new simply connected spin $4$-manifolds with $\infty^2$-property.  

We now focus our attention to refining (\ref{eq: Lambda inequality}).  
Let us define 
\begin{equation*}
\xi(p)=\frac{1}{8}\sigma(W_{12p})
=1152 p^5 -1152 p^3 -576 p^2 +4p +8.
\end{equation*} 
We can easily check that 
$$\xi'(x)=4(1440x^4-864x^2-288x+1)>0$$ 
when $x\geq 1$, and thus $\xi(p)$ is an increasing function of prime numbers $p\geq 5$.   

Let $p_k$ denote the $k$-th prime number so that $p_1=2$, $p_2=3$, etc.  
For any positive even integer $s$\/ that is greater than 
$\xi(5)=3441628$, 
we define 
\begin{equation}\label{eq: d(s)}
d(s) = p_{k+1} - p_k,
\end{equation}
where $p_k$ and $p_{k+1}$ are the unique consecutive prime numbers that satisfy:
\begin{equation}\label{eq: s interval}
\xi(p_k)<s\leq \xi(p_{k+1}). 
\end{equation}
In other words, $p_{k+1}$ is the smallest prime number such that 
$\xi^{-1}(s)\leq p_{k+1}$.  
Note that $d(s)\geq 2$.  
We also note that $s\to \infty$ if and only if $p_{k+1}\to\infty$. 
We will make use of the following theorem from \cite{bhp},
which gives the best known upper bound on the consecutive prime gap.  

\begin{theorem}\label{theorem: prime gaps}
{\rm (cf.~Theorem~1 in \cite{bhp})}
Let $p_k$ denote the $k$-th prime number.  
If\/ $\theta=0.525$, then 
there is a positive integer $N$\/ such that $p_{k+1}-p_k \leq p_{k+1}^{\theta}$ for all $k\geq N$.  
\end{theorem}

It now follows from Theorem~\ref{theorem: prime gaps} that we have 
\begin{equation}\label{eq: d(s) bound}
2\leq d(s)\leq p_{k+1}^\theta < p_{k+1}
\end{equation}
for all large enough $s$.  

\begin{corollary}\label{corollary: Lambda}
Let\/ $\Lambda_s$ be as in Definition~$\ref{definition: Lambda}$,
and let $d(s)$ be as in $(\ref{eq: d(s)})$.   
We have 
\begin{equation}\label{eq: asymptotic bound}
\Lambda_s \leq 8s
+O\big(s^{4/5}\,d(s)\big).
\end{equation}
as $s\to\infty$.
\end{corollary}

\begin{proof}
Given any even integer $s> \xi(5)$, we have $p_k=p_{k+1}-d(s)$.  
It follows from (\ref{eq: s interval}) that $\xi(p_{k+1}-d(s))< s$.  
Since $d(s)\ll p_{k+1}$ for large values of $s$\/ by (\ref{eq: d(s) bound}), we have  
\begin{align}
s> \xi(p_{k+1}-d(s))&=1152 (p_{k+1}-d(s))^5 -1152 (p_{k+1}-d(s))^3 -\cdots
\label{eq: s lower bound}\\
&= 1152(p_{k+1}^5-5p_{k+1}^4 d(s)+ 10p_{k+1}^3d(s)^2-\cdots) \nonumber \\
&\hspace{4mm}-1152(p_{k+1}^3-3p_{k+1}^2 d(s) + \cdots) - \cdots \nonumber\\
&=1152p_{k+1}^5 + O(p_{k+1}^4 d(s)).\nonumber
\end{align}
Hence for large values of $s$, we have 
\begin{equation*}
 p_{k+1}^5<\xi(p_{k+1}-d(s))< s.  
\end{equation*}
Thus we can conclude that 
\begin{equation}\label{eq: p_k+1 big-O}
p_{k+1}=O(s^{1/5})
\end{equation}
as $s\to\infty$.

It also follows from (\ref{eq: s interval}) that  
\begin{equation}\label{eq: s upper bound}
s \leq \xi(p_{k+1}) < 1152 p_{k+1}^5-1152 p_{k+1}^3 -\cdots 
< 1152 p_{k+1}^5.
\end{equation}
Combining the lower bound (\ref{eq: s lower bound}) and the upper bound (\ref{eq: s upper bound}) for $s$, we deduce that  
\begin{equation}\label{eq: s big-O}
s=1152 p_{k+1}^5+O(p_{k+1}^{4}d(s))
\end{equation}
as\/ $p_{k+1}\to\infty$.  

It now follows from (\ref{eq: c_1^2(W_n)}), (\ref{eq: p_k+1 big-O}) and (\ref{eq: s big-O}) that 
\begin{align*}
c_1^2(W_{12p_{k+1}})&= 82944 p_{k+1}^{5}+O(p_{k+1}^{3})
=\frac{82944}{1152}\cdot s +O(p_{k+1}^4 d(s))+O(p_{k+1}^{3})\\[2pt]
&= 72s+O(p_{k+1}^4 d(s)) = 72s  +O\big(s^{4/5} \, d(s)\big)
\end{align*}
as $p_{k+1}\to\infty$ and $s\to\infty$.
Substituting this last expression into (\ref{eq: Lambda inequality}) 
with $n=12p_{k+1}$, we obtain (\ref{eq: asymptotic bound}).  
\end{proof}

To refine (\ref{eq: asymptotic bound}), we need an upper bound on $d(s)$.  
To that end, we will assume the following well-known open conjecture in number theory.  
It was originally formulated by Cram\'er (cf.~\cite{cramer} and \cite{granville}) 
and has been verified for prime numbers up to $10^{18}$ (see p.~1473 of \cite{bft}).  
\begin{conjecture}[Cram\'er]\label{conjecture: Cramer}
If $p_k$ denotes the $k$-th prime number,   
then $$p_{k+1}-p_k = O((\log p_k)^2).$$  
\end{conjecture}

\begin{proposition}\label{proposition: new upper bound}
If Cram\'er's conjecture is true, then we have 
\begin{equation}\label{eq: log asymptotic bound}
\Lambda_s \leq 8s + O\big(s^{4/5} (\log s)^2 \big) 
\end{equation}
as $s\to\infty$.
\end{proposition}

\begin{proof}
Since $\log p_k <\log p_{k+1}$, we have $d(s)=p_{k+1}-p_k=O((\log p_{k+1})^2)$.
From (\ref{eq: p_k+1 big-O}), we deduce that 
$d(s)=O((\log (s^{1/5}))^2)=O((\log s)^2)$.
Substituting this into (\ref{eq: asymptotic bound}), we obtain (\ref{eq: log asymptotic bound}).  
\end{proof}

\begin{remark}\label{remark: alternative bound}
Since $s^{4/5} (\log s)^2<s^{6/7}$, the asymptotic upper bound (\ref{eq: log asymptotic bound}) is strictly stronger than the upper bound (\ref{eq: apu bound}) from \cite{apu}. 
In general, 
for the asymptotic upper bound (\ref{eq: asymptotic bound}) to be strictly stronger than (\ref{eq: apu bound}), we would need 
$s^{4/5}\,d(s)<s^{6/7}$, which is equivalent to 
$$d(s)<s^{6/7-4/5}=s^{2/35}=O(p_{k+1}^{2/7})$$
by (\ref{eq: s upper bound}).  
Thus we could certainly improve the upper bound (\ref{eq: apu bound}) if we had  
\begin{equation}\label{eq: weaker gap conjecture}
p_{k+1}-p_k = O(p_{k+1}^{\epsilon}),  
\end{equation}
for some constant $\epsilon < 2/7$.  
As far as the authors know, it is unknown whether (\ref{eq: weaker gap conjecture}) holds for some constant $\epsilon < 2/7$.  
However, we note that (\ref{eq: weaker gap conjecture}) is strictly weaker than Cram\'er's conjecture, and we observe that $\theta=0.525>2/7$ in Theorem~\ref{theorem: prime gaps}.  
\end{remark}

\section{Relation to Bogomolov-Miyaoka-Yau inequality}
\label{section: BMY inequality}

The previous section was preoccupied with providing new upper bounds on $\Lambda_s$. 
In this section, we will discuss a conjectural lower bound on $\Lambda_s$.  
The famous Bogomolov-Miyaoka-Yau inequality (cf.~\cite{bogomolov}, \cite{miyaoka}, \cite{yau}) states that 
if $M$\/ is a complex surface of general type, then we must have 
$c_1^2(M)\leq 9\chi_h(M)$, or equivalently we must have 
\begin{equation}\label{eq: bmy}
c_1^2(M)\leq 3e(M).
\end{equation}
Fintushel and Stern have conjectured that (\ref{eq: bmy}) holds for all simply connected symplectic $4$-manifolds as well 
(cf.~Conjecture~2.1 in \cite{aim} and Problem~4.90(a) in \cite{bkr}).  
At this moment, this ``symplectic BMY'' conjecture is open, i.e., we do not know any simply connected symplectic $4$-manifold $M$\/ that violates (\ref{eq: bmy}).  

If we drop the simply connected condition, then there are counterexamples.  
If $\Sigma_g$ denotes a genus $g$ curve with $g>1$, then the product ruled surface $\mathbb{CP}^1\times \Sigma_g$ is symplectic, but satisfies 
$c_1^2(\mathbb{CP}^1\times \Sigma_g)=-8(g-1)$ and $e(\mathbb{CP}^1\times \Sigma_g)=-4(g-1)$,
thereby violating (\ref{eq: bmy}).  
(See \cite{capovilla} for some more counterexamples.)  
By~(\ref{eq: limit of c_1^2/c_2}) and (\ref{eq: below BMY line}) in Section~\ref{section: symplectic 4-manifolds}, 
we know that the family of symplectic $4$-manifolds 
$$\{W_n\mid n=12p \mathrm{\ for\ primes\ }p\geq 5\}$$ approaches the BMY line $c_1^2=3e$ from below.  

If the above symplectic BMY conjecture is true, then we may deduce the following lower bound for $\Lambda_s$.  

\begin{proposition}\label{proposition: lower bound}
If the Bogomolov-Miyaoka-Yau inequality\/ $(\ref{eq: bmy})$ holds for every simply connected spin symplectic $4$-manifold $M$, then\/ 
$\Lambda_s \geq 8s-1$\/ for every even integer $s\geq 0$.
\end{proposition}

\begin{proof}
Given a nonnegative even integer $s$, let $M_{s}$ be a simply connected symplectic $4$-manifold such that $\sigma(M_{s})=8s$ and $e(M_{s})=2+8s+2\Lambda_s$.  
Such $M_{s}$ exists by (\ref{eq: e and sigma}) and Definition~\ref{definition: Lambda}.
It follows from (\ref{eq: bmy}) that 
\[
3e(M_s)\geq {c_1^2(M_{s})}=2e(M_{s})+3\sigma(M_{s}).
\]
Thus we conclude that 
\[
2+8s+2\Lambda_s=e(M_s)\geq 3\sigma(M_s)=24 s,
\]
which in turn implies that $\Lambda_s\geq 8s -1$.
\end{proof}

By combining the lower bound in the previous proposition with inequality (\ref{eq: apu bound}) from \cite{apu}, we obtain the following.  
	
\begin{corollary}\label{corollary: bmy}
If the Bogomolov-Miyaoka-Yau inequality $(\ref{eq: bmy})$ holds for every simply connected spin symplectic $4$-manifold $M$, then 
\[
\lim_{s\to\infty}\frac{\Lambda_s}{s} = 8. 
\]
\end{corollary}

In light of Corollary~\ref{corollary: bmy}, we are tempted to make the following optimistic conjecture.  

\begin{conjecture}\label{conjecture: new}
For every even integer $s\geq 0$, we have $\Lambda_s=|8s-1|$.  
\end{conjecture}

Conjecture~\ref{conjecture: new} is open for all values of $s$ at the moment.  
For example, when $s=2$, Conjecture~\ref{conjecture: new} postulates that $\Lambda_2=15$, i.e., 
the form $2E_8\oplus qH$\/ has $\infty^2$-property for every odd integer $q\geq 15$.  
By Theorem~4 in \cite{ap: spin2}, we have an upper bound $\Lambda_2\leq 155$, i.e., we know that 
$2E_8\oplus qH$\/ has $\infty^2$-property for every odd integer $q\geq 155$.  
Going back to the example $\overline{E(2)}$ in the introduction, 
if $I_{\,\overline{E(2)}}=2E_8\oplus 3H$ had $\infty^2$-property, then we would have infinitely many pairwise nondiffeomorphic counterexamples to the symplectic BMY conjecture since $c_1^2(\overline{E(2)})/e(\overline{E(2)})=96/24=4>3$.


\begin{thebibliography}{99}

\bibitem{aim}
AIM Problem Lists: \textit{Symplectic four-manifolds through branched coverings}, available at http://aimpl.org/symplecticfour/2/.

\bibitem{ap: odd} 
A. Akhmedov and B. D. Park, 
\textit{Exotic smooth structures on small\/ $4$-manifolds with odd signatures}, 
Invent. Math. \textbf{181} (2010), 577--603.  

\bibitem{ap: spin}  
A. Akhmedov and B. D. Park, 
\textit{Geography of simply connected  spin symplectic\/ $4$-manifolds},
Math. Res. Lett. \textbf{17} (2010), 483--492.

\bibitem{ap: spin2}  
A. Akhmedov and B. D. Park, 
\textit{Geography of simply connected  spin symplectic\/ $4$-manifolds, II},
C. R. Math. Acad. Sci. Paris \textbf{357} (2019), 296--298.

\bibitem{apu} 
A. Akhmedov, B. D. Park and G. Urz\'{u}a,
\textit{Spin symplectic\/ $4$-manifolds near Bogomolov-Miyaoka-Yau line},
J. G\"{o}kova Geom. Topol. GGT \textbf{4} (2010), 55--66.

\bibitem{bhp}
R. C. Baker, G. Harman and J. Pintz, 
\textit{The difference between consecutive primes}, \textit{II},
Proc. London Math. Soc. \textbf{83} (2001), 532--562.

\bibitem{bft}
W. Banks, K. Ford, and T. Tao, 
\textit{Large prime gaps and probabilistic models},
Invent. Math. \textbf{233} (2023), 1471--1518.

\bibitem{bkr}
R. \.I. Baykur, R. C. Kirby and D. Ruberman,
\textit{K}\/3:~\textit{A New Problem List in Low-Dimensional Topology},
Math. Surveys Monogr., Vol. 295, 
Amer. Math. Soc., Providence, 2026.

\bibitem{bogomolov}  
F. A. Bogomolov, 
\textit{Holomorphic tensors and vector bundles on projective manifolds}, 
Math. USSR-Izv. \textbf{13} (1979), 499--555.	

\bibitem{capovilla}
P. Capovilla,
\textit{Aspherical\/ $4$-manifolds with positive Euler characteristic and their geography}, 
arXiv:2511.15577, 2025.

\bibitem{cramer}
H. Cram\'er,
\textit{On the order of magnitude of the difference between consecutive prime numbers},
Acta Arith. \textbf{2} (1936), 23--46.

\bibitem{ev}
H. Esnault and E. Viehweg, 
\textit{Lectures on vanishing theorems},
DMV Sem., Vol.~20, 
Birkh\"auser Verlag, Basel, 1992. 

\bibitem{fs}  
R. Fintushel and R. J. Stern, 
\textit{Knots, links and\/ $4$-manifolds}, 
Invent. Math. \textbf{134} (1998), 363--400.

\bibitem{fm}  
A. T. Fomenko and S. V. Matveev, 
\textit{Algorithmic and Computer Methods for Three-Manifolds}, 
Mathematics and its Applications, Vol. 425, 
Kluwer Academic Publishers, Dordrecht, 1997. 

\bibitem{freedman}  
M. H. Freedman,  
\textit{The topology of four-dimensional manifolds}, 
J. Differential Geom. \textbf{17} (1982), 357--453.  

\bibitem{gompf}  
R. E. Gompf, 
\textit{A new construction of symplectic manifolds},
Ann. of Math. \textbf{142} (1995), 527--595.

\bibitem{gompf-mrowka}  
R. E. Gompf and T. S. Mrowka, 
\textit{Irreducible\/ $4$-manifolds need not be complex}, 
Ann. of Math. \textbf{138} (1993), 61--111. 

\bibitem{gompf-stipsicz}  
R. E. Gompf and A. I. Stipsicz, 
\textit{$4$-Manifolds and Kirby Calculus}, 
Grad. Stud. Math., Vol.~20, Amer. Math. Soc., Providence, 1999.

\bibitem{granville}
A. Granville,
\textit{Harald Cram\'er and the distribution of prime numbers},
Harald Cram\'er Symposium,
Scand. Actuar. J. \textbf{1} (1995), 12--28. 

\bibitem{hirzebruch}
F. Hirzebruch, 
\textit{Topological Methods in Algebraic Geometry},
Die Grundlehren der mathematischen Wissenschaften, Vol.~131,
Springer-Verlag, New York, 1966.

\bibitem{mccarthy-wolfson}  
J. D. McCarthy and J. G. Wolfson,
\textit{Symplectic normal connect sum}, 
Topology \textbf{33} (1994), 729--764. 

\bibitem{miyaoka}  
Y. Miyaoka, 
\textit{On the Chern numbers of surfaces of general type}, 
Invent. Math. \textbf{42} (1977), 225--237.

\bibitem{ps}  
B. D. Park and Z. Szab\'o,  
\textit{The geography problem for irreducible spin four-manifolds}, 
Trans. Amer. Math. Soc. \textbf{352} (2000), 3639--3650.

\bibitem{jpark: spin}  
J. Park, 
\textit{The geography of spin symplectic\/ $4$-manifolds}, 
Math. Z. \textbf{240} (2002), 405--421. 

\bibitem{rohlin}  
V. A.  Rohlin, 
\textit{New results in the theory of four-dimensional manifolds}, 
Doklady Akad. Nauk SSSR \textbf{84} (1952), 221--224. 

\bibitem{ru}
X. Roulleau and G. Urz\'ua, 
\textit{Chern slopes of simply connected complex surfaces of general type are dense in\/ $[2,3]$},
Ann. of Math. \textbf{182} (2015), 287--306.

\bibitem{yau}  
S.-T. Yau, 
\textit{Calabi's conjecture and some new results in algebraic geometry},
Proc. Nat. Acad. Sci. U.S.A. \textbf{74} (1977), 1798--1799.  

\end{thebibliography}
\end{document}